\documentclass[fleqn]{zzz}
\usepackage{mathrsfs}
\usepackage{color}
\usepackage{hyperref}

\hypersetup{
   colorlinks = true,
   urlcolor = blue,
   linkcolor = red
}

\DeclareRobustCommand{\stirling}{\genfrac\{\}{0pt}{}}

\newtheorem{thm}[lemma]{Theorem}
\newtheorem{cor}[lemma]{Corollary}
\newtheorem{prop}[lemma]{Proposition}

\numberwithin{equation}{section}

\begin{document}

\ArticleName{Base-$b$ analogues of classic combinatorial objects}

\Author{Tanay V. Wakhare $^\ast$, Christophe Vignat $^\dag$}
\AuthorNameForHeading{T.~V.~Wakhare, C.~Vignat}

\Address{$^\ast$~University of Maryland, College Park, MD 20742, USA}
\EmailD{twakhare@gmail.com}

\Address{$^\dag$~Tulane University, New Orleans, LA 70118, USA}
\EmailD{cvignat@tulane.edu}

%\ArticleDates{Received XX July 2016 in final form ????; Published online ????}

\Abstract{We study the properties of the base-$b$ binomial coefficient defined by Jiu and the second author, introduced in the context of a digital binomial theorem. After introducing a general summation formula, we derive base-$b$ analogues of the Stirling numbers of the second kind, the Fibonacci numbers and the classical exponential function.}

\Classification{05A10;11B65;11B73;11B39}
%http://www.ams.org/mathscinet/msc/msc2010.html
%TODO tab convention
%TODO keywords and classification
%TODO explicitly define natural numbers as including 0?
\section{Introduction}
\label{Introduction}

Recently, Jiu and the second author have conducted work \cite{Vignat1} concerning the base-$b$ binomial coefficient. We let $n=\sum_{i=0}^{N_n-1}n_i b^i$ and $k=\sum_{i=0}^{N_k-1}k_i b^i$ be the base-$b$ expansions of $n$ and $k$ respectively. Then we define $N:=\max\{N_n, N_k\}$, the base-$b$ binomial coefficient is given as the product
\begin{equation}\label{1.1}
\binom{n}{k}_b := \prod_{i=0}^{N-1}\binom{n_i}{k_i},
\end{equation}
%and
so that $\binom{n}{k}_b=0$ if $k_i>n_i$ for some $i.$ 

This also motivates the definition of a base $b$ factorial, which we define as 
\[(n!)_b:=\prod_{i=0}^{N_n-1}n_i!.\]

The binomial theorem is a classic result stating  that $(X+Y)^{n} = \sum_{k=0}^{n} \binom{n}{k} X^{k} Y^{n-k}$. Nguyen subsequently generalized this to the digital binomial theorem in \cite[(26)]{Nguyen1}, which states that
\begin{equation}\label{1.2}
(X+Y)^{s_2(n)} = \sum_{0\leq k \leq_2 n} X^{s_2(k)} Y^{s_2(n-k)},
\end{equation}
where $s_2(n)$ denotes the sum of the digits of $n$ expressed in base $2$. The condition $k \leq_2 n$ restricts the summation over indices such that each digit of $k$ is less than the corresponding digit of $n$ in base $2$, i.e. $k_i\leq n_i$ for all $0\leq i \leq N-1$. This notation, which has previously been referred to as digital dominance, will be adopted throughout this paper. This is also equivalent to the condition that the addition of $n-k$ and $k$ be carry-free in base $2$, which can be written symmetrically as $s_2(k)+s_2(n-k)=s_2(n)$ so that an equivalent form of \eqref{1.2} reads
\[
(X+Y)^{s_2(n)} = \sum_{s_2(k)+s_2(n-k)=s_2(n)} X^{s_2(k)} Y^{s_2(n-k).}
\]

The discovery of this digital binomial theorem spurred further extensions by Nguyen \cite{Nguyen1} \cite{Nguyen2}, Nguyen and Mansour \cite{Mansour1,Mansour2} and Liu and the second author \cite{Vignat1}, which led to the introduction of the base-$b$ binomial coefficient $\binom{n}{k}_b$ and a generalization of the digital binomial theorem to an arbitrary base $b$ as follows:

\[
(X+Y)^{s_b(n)} = \sum_{k=0}^{n} \binom{n}{k}_b X^{s_b(k)} Y^{s_b(n-k)}
\]
where $s_b\left(n\right)$ is the sum of the digits of $n$ in base $b$ and the base-$b$ binomial coefficient is
given by \eqref{1.1}.

The paper is organized as follows. In Section \ref{CarryfreeSums}, we present a general summation formula, which is used in Section \ref{Stirling} to derive theorems about a base-$b$ analogue of the Stirling numbers of the second kind. This same formula is used to define a base-$b$ analogue of the Fibonacci numbers in Section \ref{Fibonacci}. Finally, in Section \ref{Exponential} we introduce an analogue of the exponential function involving the base-$b$ factorial and derive some of its properties.

\section{Sums over carry-free k}
\label{CarryfreeSums}
We can extend the methods of \cite{Vignat1} to take sums over carry-free $k$, with a weighting by the base-$b$ binomial coefficient. 

\begin{thm}\label{thm1}
Let $S(n):=\sum_{k=0}^{n} f(n,k)$. Then
\begin{equation}\label{2.1}
\prod_{i=0}^{N-1}S\left(n_i\right) = \sum_{0\leq k \leq_b n} \prod_{i=0}^{N-1}f(n_i, k_i).
\end{equation}
\end{thm}
\begin{proof}
We have
\begin{align*}
\prod_{i=0}^{N-1}S\left(n_i\right) &= \sum_{k_0=0}^{n_0}f\left(n_0,k_0\right)\sum_{k_1=0}^{n_1}f\left(n_1,k_1\right) \cdots\sum_{k_{N-1}=0}^{n_{N-1}}f\left(n_{N-1},k_{N-1}\right)\\
&=\sum_{k_0=0}^{n_0}\sum_{k_1=0}^{n_1}\cdots\sum_{k_{N-1}=0}^{n_{N-1}}f\left(n_0,k_0\right)f\left(n_1,k_1\right)\cdots
f\left(n_{N-1},k_{N-1}\right)\\
&=\sum_{0\leq k \leq_b n} \prod_{i=0}^{N-1} f\left(n_i,k_i\right).\end{align*}
\end{proof}

\begin{cor}
Let $S_2(n):=\sum_{k=0}^{n} \binom{n}{k}f(n,k)$. Then
\begin{equation}\label{2.3}
\prod_{i=0}^{N-1}S_2\left(n_i\right) = \sum_{0\leq k \leq_b n} \prod_{i=0}^{N-1}\binom{n_i}{k_i}f(n_i,k_i) = \sum_{k=0}^{n} \binom{n}{k}_b \prod_{i=0}^{N-1}f(n_i,k_i).
\end{equation}

\begin{proof}
Replace $f\left(n,k\right)$ with $\binom{n}{k}f\left(n,k\right)$ in \eqref{2.1}. We can extend the sum over all $0\leq k \leq n$ since $\binom{n}{k}_b=0$ if $k$ is not digitally dominated by $n$.\end{proof}
\end{cor}

While the algebraic proof of this identity is very simple, it greatly simplifies and supersedes the proof of many previously discovered identities relating to sums over digitally dominated $k$, and admits much more sweeping generalizations. The previous work done on sums over carry-free $k$ has centered on proving this identity for special values of $f(n,k%,i
)$, often using the properties of infinite matrices. We now give the choices of $f\left(n,k%,i
\right)$ that yield the central theorems in previous papers \cite{Mansour1} \cite{Mansour2} \cite{Nguyen1} \cite{Nguyen2} \cite{Vignat1}.
\begin{itemize}
\item
By taking $f(n,k%,i
) = X^k Y^{n-k}$ in \eqref{2.3} with constant $X$ and $Y$, and noting that $S_2(n%,i
) = \sum_{k=0}^{n}\binom{n}{k} X^k Y^{n-k} = (X+Y)^{n}$, we recover the base-$b$ binomial theorem \cite[(Theorem 2)]{Vignat1}. Taking the $b=2$ case reduces to the digital binomial theorem \cite[(26)]{Nguyen1}. Namely,
\begin{equation}\label{2.4}
\left(X+Y\right)^{s_b(n)} = \sum_{k=0}^{n}\binom{n}{k}_b X^{s_b(k)}Y^{s_b(n-k)}.
\end{equation}
\item
We take $f(n,k%,i
) = \frac{\left(X\right)_k}{k!}\frac{\left(Y\right)_{n-k}}{\left(n-k\right)!}$ with constant $X$ and $Y$ in \eqref{2.1}, where $(x)_k := \frac{\Gamma(x+k)}{\Gamma(x)}$ is the Pochhammer symbol. Noting that $S(n%,i
) = \frac{\left(X+Y\right)_n}{n!}$ by the Vandermonde identity, we obtain an extension of the base-$b$ binomial theorem, the central theorem in \cite[(Theorem 2)]{Nguyen2}. Namely,
\begin{equation}\label{2.5}
\prod_{i=0}^{N-1}\binom{X+Y+n_i-1}{n_i} =\sum_{0\leq k \leq_b n}\prod_{i=0}^{N-1} \binom{X+k_i-1}{k_i} \binom{Y+n_i-k_i-1}{k_i} .
\end{equation}

%We take $f(\vec{a}) = \binom{x;r}{k}\binom{y;r}{n-k}$ in \eqref{2.1} where $\binom{x;r}{d}:=\frac{x(x+r)\cdots(x+(d-1)r)}{d!}$, and noting that $S(x,y,r,n,k) = \binom{x+y;r}{n}$\cite{Mansour}[(Lemma 9)], we obtain a generalized theorem that yields many $q$-analogues of the base-$b$ binomial theorem. Namely,
%\begin{equation}
%\prod_{i=0}^{N-1}\binom{x_i+y_i;r}{n_i} =\sum_{0\leq k \leq_b n}\prod_{i=0}^{N-1} \binom{x_i;r_i}{k_i} \binom{y_i;r_i}{n_i-k_i} .
%\end{equation}

\item
Taking $f\left(n,k%,i
\right) = \binom{x;r}{k}\binom{y;r}{n-k}$ in \eqref{2.1} where $\binom{x;r}{d}:=\frac{x(x+r)\cdots(x+(d-1)r)}{d!}$ \cite[(Definition 8)]{Mansour1} and using \cite[(Lemma 9)]{Mansour1} to find the sum of $f$ over $k$ gives \cite[(Theorem 4)]{Mansour1}, which can be specialized to find $q$-analogues. Namely,
\begin{equation}\label{2.6}
\prod_{i=0}^{N-1}\binom{x_i+y_i;r}{n_i} =\sum_{0\leq k \leq_b n}\prod_{i=0}^{N-1} \binom{x_i;r_i}{k_i} \binom{y_i;r_i}{n_i-k_i} .
\end{equation}

\item
Taking $f\left(n,k%,i
\right) =  \overline{p}_k(x)\overline{s}_{n-k}(y)$ and $f(n,k%,i
) = \overline{p}_k(x)\overline{p}_{n-k}(y)$ in \eqref{2.1} where $s$ and $p$ are normalized Sheffer sequences and using \cite[(Theorem 7)]{Mansour2} to evaluate the sum of $f$ over $k$ gives \cite[(Theorem 1)]{Mansour2}. 
\end{itemize}

As a corollary of \eqref{2.6}, we obtain a base-$b$ analog of the Chu-Vandermonde identity. The Chu-Vandermonde identity states \cite[(Page 8)]{Riordan} that, for integer $m$ and $n$,
\begin{equation}\label{2.7}
\sum_{k=0}^{r}\binom{n}{k}\binom{m}{r-k}=\binom{n+m}{r}.
\end{equation}
This can be extended to complex valued $n$ and $m$. This has a base-$b$ counterpart \cite[(Theorem 8)]{Vignat1} which can be proved by taking $r=1$ in \eqref{2.6} and applying \eqref{1.1}. We note that if the addition of $n$ and $m$ in base $b$ in carry-free, then $\prod_{i=0}^{N-1}\binom{n_i+m_i}{r_i} = \binom{n+m}{r}_b$. Changing variables to maintain consistent notation yields the following analog of the Chu-Vandermonde identity.
\begin{thm} For integral $n$ and $m$ such that their addition in base $b$ in carry-free,
\begin{equation}
\binom{n+m}{r}_b=\sum_{0 \leq k \leq_b r}\binom{n}{k}_b\binom{m}{r-k}_b.
\end{equation}
\end{thm}
As with the digital binomial theorem, an identity involving binomial coefficients has an obvious base-$b$ analog. However, we emphasize that the results do not immediately transfer, as can be seen by the restriction on $n$ and $m$ in the base-$b$ Chu-Vandermonde identity. However, the striking similarities between  the digital binomial theorem and binomial theorem ensure that some linear identities transfer almost identically. We list some  analogs of classic binomial identities derived from the binomial theorem, omitting the proofs since they follow classical lines. 

\begin{equation}
\sum_{k=0}^{n}\binom{n}{k}_b  = 2^{s_b(n)}.
\end{equation}
\begin{equation}
\sum_{k=0}^{n}\binom{n}{k}_b s_b(k) = s_b(n)2^{s_b(n)-1}.
\end{equation}
\begin{equation}
\sum_{k=0}^{n}\binom{n}{k}_b {s_b}^2(k) =s_b(n)(s_b(n)+1)2^{s_b(n)-2}.
\end{equation}

\section{Generating Function}\label{genfunc}
From the Theorem \ref{thm1} we can also derive a generating function for digitwise functions, based on the following lemma.
\begin{lemma}\label{genlem}
\begin{equation}
\{k: 0\leq k \leq_b n, 0 \leq n_i \leq b-1, 0 \leq i < \infty \} = \{k: 0 \leq k \leq \infty \}.
\end{equation}
\begin{proof}
There is a bijection between both sets. To see this, fix a base $b$. Every natural number has a unique representation in base $b$ which occurs in the left-most set, while every number in the left-most set corresponds to a unique natural number $k$.
\end{proof}
\end{lemma}
We now state without proof the straightforward extension of \eqref{2.1}: let  $S(n,i):=\sum_{k=0}^{n} f(n,k,i)$. Then
\begin{equation}\label{ext}
\prod_{i=0}^{N-1} S\left(n_i,i\right) = \sum_{0\leq k \leq_b n} \prod_{i=0}^{N-1}f(n_i, k_i,i).
\end{equation}
We then extend the summation over digitally dominated $k$ in \eqref{ext} to a sum over all natural numbers $k$. This means by letting the number of digits in \eqref{ext} tend to infinity while letting each $n_i = b-1$ and applying Lemma \ref{genlem} we can convert our sum over digitally dominated $k$ to a sum over natural numbers.

\begin{thm}\label{thm2}
\begin{equation}
\label{eqthm2}
\prod_{i=0}^{\infty} \sum_{k=0}^{b-1} f(k,i) = \sum_{k=0}^{\infty} \prod_{i=0}^{\infty} f(k_i,i).
\end{equation}
\begin{proof}
Fixing a base $b$ and real $x$ and setting $n_i=b-1$ in \eqref{ext} while taking $N \rightarrow \infty$, we can apply the preceeding lemma to \eqref{ext} as
\begin{equation}
\prod_{i=0}^{\infty} \sum_{k_i=0}^{b-1} f(k_i,i) = \sum_{k=0}^{\infty} \prod_{i=0}^{\infty} f(k_i,i).
\end{equation}
We can remove the dependence of $k$ on $i$ on the left hand side since each $k$ goes from $0$ to $b-1$ identically, which completes the proof.
\end{proof}
\end{thm}

\begin{cor}\label{genthm}
\begin{equation}
\prod_{i=0}^{\infty} \sum_{k=0}^{b-1} f(k,i) x^{b^i k} = \sum_{k=0}^{\infty}x^k \prod_{i=0}^{\infty} f(k_i,i).
\end{equation}
\begin{proof}
Making the substitution $f(k_i,i) = x^{b^i k_i} f(k_i,i)$ in \eqref{eqthm2} and noting
\begin{equation}
\prod_{i=0}^{\infty} x^{b^i k_i} f(k_i,i) = x^{\sum_{i=0}^{\infty}b^i k_i } \prod_{i=0}^{\infty} f(k_i,i) = x^k  \prod_{i=0}^{\infty} f(k_i,i)
\end{equation}
completes the proof.
\end{proof}
\end{cor}
This means that any base-$b$ function which can be represented as a digitwise product has a closed form generating function.

\section{Stirling numbers of the second kind}
\label{Stirling}
The Stirling numbers of the second kind $\stirling{n}{k}$ count the number of ways to partition a set of $n$ elements into $k$ non-empty subsets. They have a natural base-$b$ analog, defined as
\begin{equation}
\label{Stirling_definition}
\stirling{n}{k}_{b}= \prod_{i=0}^{N-1} \stirling{n}{k_i}.
%S_b(n,k) = \prod_{i=0}^{N-1}S(n,k_i).
\end{equation}
\begin{thm}
The base-$b$ Stirling numbers of the second kind satisfy
\begin{equation}\label{3.1}
\stirling{n}{k}_{b} = \frac{1}{(k!)_b}\sum_{j=0}^k (-1)^{s_b(k)-s_b(j)}\binom{k}{j}_b \left(\prod_{i=0}^{N-1}j_i\right)^n.
%S_b(n,k) = \frac{1}{(k!)_b}\sum_{j=0}^k (-1)^{s_b(k)-s_b(j)}\binom{k}{j}_b \left(\prod_{i=0}^{N-1}j_i\right)^n.
\end{equation}
\begin{proof}
From \cite[(Page 90)]{Riordan} the Stirling numbers of the second kind can be represented as
\begin{equation}\label{3.2}
%S(n,k) = \frac{1}{k!}\sum_{j=0}^{k}(-1)^{k-j}\binom{k}{j}j^n.
\stirling{n}{k} = \frac{1}{k!}\sum_{j=0}^{k}(-1)^{k-j}\binom{k}{j}j^n.
\end{equation}
Rearranging this relation and changing variables to maintain consistency with \eqref{2.3} yields
\begin{equation}\label{3.3}
%\sum_{k=0}^{n}(-1)^{k}\binom{n}{k}k^\alpha = (-1)^n n! S(\alpha,n).
\sum_{k=0}^{n}(-1)^{k}\binom{n}{k}k^\alpha = (-1)^n n! \stirling{\alpha}{n}.
\end{equation}
Taking $f(n,k)=(-1)^k k^\alpha$ in \eqref{2.3} and noting that $S_2(n)$ is nothing more than the left hand side of \eqref{3.2}, we obtain
\begin{equation}\label{3.4}
\prod_{i=0}^{N-1}(-1)^{n_i}n_i!%S(\alpha,n_i)
\stirling{\alpha}{n_i} = \sum_{k=0}^n \binom{n}{k}_b \prod_{i=0}^{N-1}(-1)^{k_i} {k_i}^\alpha.
\end{equation}
Simplifying the product on both sides yields
\begin{equation}\label{3.5}
(-1)^{s_b(n)}(n!)_b\prod_{i=0}^{N-1}%S(\alpha,n_i) 
\stirling{\alpha}{n_i}= \sum_{k=0}^n \binom{n}{k}_b (-1)^{s_b(k)}\left(\prod_{i=0}^{N-1}k_i\right)^\alpha.
\end{equation}
Rearranging then gives
\begin{equation}\label{3.6}
\prod_{i=0}^{N-1}%S(\alpha,n_i) 
\stirling{\alpha}{n_i}= \frac{1}{(n!)_b}\sum_{k=0}^n (-1)^{s_b(n)-s_b(k)}\binom{n}{k}_b \left(\prod_{i=0}^{N-1}k_i\right)^\alpha,
\end{equation}
which can be compared with \eqref{3.2}. The parallels are immediately obvious. Renaming variables completes the proof.
\end{proof}
\end{thm}

In general, any sequence of numbers with an explicit representation as a sum involving a binomial coefficient will have a base-$b$ analog. But this representation is always possible, since an arbitrary sequence $\left(a_n\right)$ can always expressed as \cite[p.43]{Riordan}
\[
a_n = \sum_{k=0}^{n} \binom{n}{k}\left(-1\right)^k b_k
\]
with the sequence $\left(b_n\right)$ defined by
\[
b_n = \sum_{k=0}^{n} \binom{n}{k}\left(-1\right)^k a_k.
\]
These base-$b$ Stirling numbers can then be used to generalize certain results about the Stirling numbers. A fundamental result concerning the differential operator $\vartheta:=x D$, $D:=\frac{d}{dx}$, found in \cite[(Page 218)]{Riordan}, is
\begin{equation}\label{3.7}
\vartheta^n = \sum_{k=0}^{n}%S(n,k)
\stirling{n}{k}x^k D^k.
\end{equation}

\begin{thm}
Define a multivariable differential operator $\vartheta_{N}:=x_0x_1\cdots x_{N-1}D_0 D_1 \cdots D_{N-1}$ where $D_i:=\frac{\partial}{\partial x_i}$. Then
\begin{equation}\label{3.8}
\vartheta_{N}^n = \sum_{k_1,\ldots, k_{N-1}\leq n}% S_b(n,k)
\stirling{n}{k}_{b}\prod_{i=0}^{N-1}x_i^{k_i}D_i^{k_{i}}.
\end{equation}
\begin{proof}
By the equality of mixed partial derivatives,
\begin{equation}\label{3.9}
\vartheta_{N}^n = \left(x_0D_0x_1D_1\cdots x_{N-1}D_{N-1}\right)^n = (x_0D_0)^n(x_1D_1)^n \cdots (x_{N-1}D_{N-1})^n.
\end{equation}
Expanding each term using \eqref{3.7} gives
\begin{align}
\vartheta_{N}^n &= \left(\sum_{k_0=0}^{n}%S(n,k_0)
\stirling{n}{k_0} x_0^{k_0} D_0^{k_{0}}\right)\cdots \left(\sum_{k_{N-1}=0}^{n}%S(n,k_{N-1})
\stirling{n}{k_{N-1}} x_{N-1}^{k_{N-1}} D_{N-1}^{k_{N-1}}\right) \nonumber \\
&=\sum_{k_1,\ldots, k_{N-1}\leq n} \prod_{i=0}^{N-1}%S(n,k_i)
\stirling{n}{k_i} x_i^{k_i}D_i^{k_{i}} \nonumber \\
&=\sum_{k_1,\ldots, k_{N-1}\leq n} %S_b(n,k)
\stirling{n}{k}_{b}\prod_{i=0}^{N-1}x_i^{k_i}D_i^{k_{i}} \label{3.10}.
\end{align}
In the second equality, we used the interchange of summation presented in \eqref{2.1}. 
\end{proof}
\end{thm}

An explicit expression of the base-$b$ Stirling numbers as sum over partitions can be obtained as follows.
\begin{thm}
\begin{equation}
%S_b(n,k)
\stirling{n}{k}_b = \sum_{j}\prod_{|J|=j}k_J
\stirling{n-1}{k_J}_b\stirling{n-1}{k_{\bar{J}}-1}_b,
\end{equation}
where the product is over all partitions $J$ of $\{1,\ldots,n\}$ and $\bar{J}=\{1,\ldots,n\}\char`\\ J$
\begin{proof}
We begin with the Pascal-type recurrence \cite[(4)]{Agoh1} %$S(n,k) = S(n-1,k-1) + kS(n-1,k)$ 
$\stirling{n}{k}=\stirling{n-1}{k-1}+k\stirling{n-1}{k}$ and substitute into the first equality in \eqref{3.1}. Namely,
\begin{equation}
%S_b(n,k) = \prod_{i=0}^{N-1}S(n,k_i) = \prod_{i=0}^{N-1}S(n-1,k_{i}-1) + k_{i}S(n-1,k_{i}),
\stirling{n}{k}_b = \prod_{i=0}^{N-1}\stirling{n}{k_i} = \prod_{i=0}^{N-1}\left( %S(n-1,k_{i}-1) + k_{i}S(n-1,k_{i}) 
\stirling{n-1}{k_{i}-1} +k_{i} \stirling{n-1}{k_{i}}\right) ,
\end{equation}
from which the theorem follows.
\end{proof}
\end{thm}

As a last remark, we start from the well-known representation of the Stirling numbers of the second
kind 
\[
\stirling{n}{k}=\frac{\left(-1\right)^{k}}{k!}\Delta^{k}x^{n}\thinspace_{\vert x=0}
\]
where $\Delta$ is the forward difference operator
\[
\Delta f\left(x\right)=f\left(x+1\right)-f\left(x\right).
\]
With $s_{b}\left(n\right)$ denoting the sum of digits of $n$ in
base $b$, we deduce from \eqref{Stirling_definition} the representation of the base $b$ Stirling
numbers as
\[
\stirling{n}{k}_{b}=\frac{\left(-1\right)^{s_{b}\left(k\right)}}{k!_{b}}\Delta^{s_{b}\left(k\right)}x^{n}\thinspace_{\vert x=0}
\]

The base-$b$ Stirling number naturally occurs as a consequence of creating a multivariable generalization of an existing single variable identity, a process which could be explored further in the future. Additionally, the base $b$ has no special significance here so long as it is larger than $k_i$, suggesting that $\stirling{n}{k}_{b}$ has a larger combinatorial significance outside of the sum of digits function.

%\iffalse
\section{Fibonacci numbers}
\label{Fibonacci}
In this section, we introduce one further analogue of classic numbers with well-studied properties. The Fibonacci numbers $F_n$ are defined by the recurrence $F_{n+1} = F_n + F_{n-1}$ and the initial values $F_0=1$ and $F_1 = 1$.

\subsection{Definition}

\begin{thm}
Define the base-$b$ generalized Fibonacci numbers $F^{\left(b\right)}_{n}$ as
\begin{equation}\label{4.0}
F^{\left(b\right)}_{n} = \prod_{i=0}^{N-1}F_{n_i}.
\end{equation}
Then these numbers satisfy
\begin{equation}\label{4.1}
F^{\left(b\right)}_{n} = \sum_{k \le_{b} n}\binom{n-k}{k}_b.
\end{equation}
%Furthermore,
%\begin{equation}\label{4.2}
%F^{b,j}_{n} = F^{b,j}_{n-b^j} +  F^{b,j}_{n-2b^j}.
%\end{equation}
\begin{proof}
By summing over shallow diagonals of Pascal's triangle, we have the relation
\begin{equation}\label{4.3}%%TODO Cite
F_n = \sum_{k=0}^{n}\binom{n-k}{k}.
\end{equation}
Taking $f(n,k)=\binom{n-k}{k}$ in \eqref{2.1} and utilizing the relation \eqref{4.3} yields \eqref{4.1}.
%From \cite[(Remark 11)]{Vignat1}, for $0\leq j \leq N-1$ we have the recurrence
%\begin{equation}\label{4.4}
%\binom{n}{k}_b = \binom{n-b^j}{k-b^j}_b + \binom{n-b^j}{k}_b.
%\end{equation}
%Applying this recurrence to \eqref{4.1} yields
%\begin{align*}
%F^{b,j}_{n} &= \sum_{k=0}^n \binom{n-k-b^j}{k-b^j}_b + \sum_{k=0}^n\binom{n-k-b^j}{k}_b \\
%& = \sum_{k=0}^{n-b^j} \binom{n-2b^j-(k-b^j)}{k-b^j}_b + \sum_{k=0}^n\binom{n-k-b^j}{k}_b \\
%& = \sum_{k=0}^{n-2b^j} \binom{n-2b^j-k}{k}_b + \sum_{k=0}^{n-b^j}\binom{n-k-b^j}{k}_b \\
%&= F^{b,j}_{n-2b^j} + F^{b,j}_{n-b^j}.
%\end{align*}
%This completes the proof of the second part of the theorem.
\end{proof}
\end{thm}
Rather than the Stirling numbers, which have combinatorial significance, the Fibonacci numbers are well studied objects in number theory. 

\subsection{The case $b=3$}

The sequence $F_{n}^{\left(3\right)}$, for $n\ge0,$ starts with  
\[
1,1,2,1,1,2,2,2,4,1,1,2,1,1,2,2,2,4,2,2,4,2,2,4,4,4,8,1,1,2\dots
\]
%\begin{table}[h]
%\begin{centering}
%\begin{tabular}{|c|c|c|c|c|c|c|c|c|c|c|c|c|c|c|c|}
%\hline 
%$n$ & $0$ & $1$ & $2$ & $3$ & $4$ & $5$ & $6$ & $7$ & $8$ & $9$ & $10$ & $11$ & $12$ & $13$ & $14$\tabularnewline
%\hline 
%\hline 
%$F_{n}^{\left(3\right)}$ & $1$ & $1$ & $2$ & $1$ & $1$ & $2$ & $2$ & $2$ & $4$ & $1$ & $1$ & $2$ & $1$ & $1$ & $2$\tabularnewline
%\hline 
%\end{tabular}
%\par\end{centering}
%\begin{centering}
%\begin{tabular}{|c|c|c|c|c|c|c|c|c|c|c|c|c|c|c|c|}
%\hline 
%$n$ & $15$ & $16$ & $17$ & $18$ & $19$ & $20$ & $21$ & $22$ & $23$ & $24$ & $25$ & $26$ & $27$ & $28$ & $29$\tabularnewline
%\hline 
%\hline 
%$F_{n}^{\left(3\right)}$ & $2$ & $2$ & $4$ & $2$ & $2$ & $4$ & $2$ & $2$ & $4$ & $4$ & $4$ & $8$ & $1$ & $1$ & $2$\tabularnewline
%\hline 
%\end{tabular}
%\par\end{centering}
%\caption{First values of $F_{n}^{\left(3\right)}.$}
%\end{table}

We recognize the beginning of sequence $A117592$ in OEIS defined by 
\[
a\left(3n\right)=a\left(n\right),\thinspace\thinspace a\left(3n+1\right)=a\left(n\right),\thinspace\thinspace a\left(3n+2\right)=2a\left(n\right)
\]
and
\[
a\left(0\right)=a\left(1\right)=1,\thinspace\thinspace a\left(2\right)=2.
\]

\begin{thm}
The sequence $\left\{ F_{n}^{\left(3\right)}\right\} $ satisfies
the recurrence
\[
F_{3n}^{\left(3\right)}=F_{n}^{\left(3\right)},\thinspace\thinspace F_{3n+1}^{\left(3\right)}=F_{n}^{\left(3\right)},\thinspace\thinspace F_{3n+2}^{\left(3\right)}=2F_{n}^{\left(3\right)}
\]
with initial conditions
\[
F_{0}^{\left(3\right)}=F_{1}^{\left(3\right)}=1,\thinspace\thinspace F_{2}^{\left(3\right)}=2
\]
so that it coincides with OEIS sequence $A117592$.

\begin{proof}
By our main theorem, the base-$b$ Fibonacci sequence satisfies
\[
F_{m}^{\left(b\right)}=\prod_{i}F_{m_{i}}.
\]
The digits of $m=3n$ (resp. $m=3n+1$ and $m=2n+2$) coincide with those
of $n$ up to an extra $0$ (resp. $1$ and $2$) appended at the right.
We deduce
\begin{equation}
\label{F3n}
F_{3n}^{\left(3\right)}=F_{n}^{\left(3\right)}F_{0}=F_{n}^{\left(3\right)}
\end{equation}
and
\begin{equation}
\label{F3n+1}
F_{3n+1}^{\left(3\right)}=F_{n}^{\left(3\right)}F_{1}=F_{n}^{\left(3\right)}
\end{equation}
and
\begin{equation}
\label{F3n+2}
F_{3n+2}^{\left(3\right)}=F_{n}^{\left(3\right)}F_{2}=2F_{n}^{\left(3\right)}.
\end{equation}
The initial conditions are easily checked.
\end{proof}
\end{thm}

As a consequence of identities \eqref{F3n}, \eqref{F3n+1} and \eqref{F3n+2}, we have
\begin{cor}
The Fibonacci numbers $F_{n}^{\left(3\right)}$ satisfy the recurrence
\[
F_{3n+2}^{\left(3\right)}=F_{3n}^{\left(3\right)}+F_{3n+1}^{\left(3\right)}
\]
which can be interpreted as a $\mod 3$ version of the usual recursion on the Fibonacci numbers
\[
F_{n+2}=F_{n+1}+F_{n}.
\]
\end{cor}

A combinatorial interpretation for the sequence $F_n^{\left(3\right)}$ is obtained using the base-$3$ expansion
of the integer $n$, which consist of a sequence of $0$'s, $1$'s and $2$'s, the number of which we call respectively 
$s_3\left(n,0\right),\,\,s_3\left(n,1\right)$ and $s_3\left(n,2\right)$. Then from definition \eqref{4.0}, we deduce
\[
F_n^{\left(3\right)} = F_{0}^{s_3\left(n,0\right)}F_{1}^{s_3\left(n,1\right)}F_{2}^{s_3\left(n,2\right)}
=2^{s_3\left(n,2\right)}
\]
so that $F_n^{\left(3\right)}$ essentially counts the number of $2$'s in the base $3$ representation of $n.$
We remark that this interpretation does not extend to the case of a base $b >3$.

\subsection{The general case}

This suggests the more general result as follows. Its proof is omitted
since it is identical to the previous one.
\begin{thm}
For $0\le p<b,$ the sequence $\left\{ F_{n}^{\left(b\right)}\right\} $
satisfies the recurrence
\[
F_{bn+p}^{\left(b\right)}=F_{n}^{\left(b\right)}F_{p}
\]
with initial conditions
\[
F_{p}^{\left(b\right)}=F_{p}.
\]
\end{thm}

We also have the following theorem:
\begin{thm}
A generating function for the sequence $\left\{ F_{n}^{\left(b\right)}\right\}$ is
\[
{\cal F}_b\left(z\right)=\sum_{n \ge 0} F_{n}^{\left(b\right)}z^n = \prod_{k \ge 0} 
\left(\sum_{l=0}^{b-1} F_l z^{b^l} \right)
\]
\begin{proof}
Take $f(k,i)=F_k$ in Theorem \ref{genfunc}.
\end{proof}
\end{thm}

Additional identities similar to those satisfied by the usual Fibonacci numbers are derived from 
the identities \eqref{F3n}, \eqref{F3n+1} and \eqref{F3n+2} as follows.
\begin{thm}
The Fibonacci numbers $F_{n}^{\left(b\right)}$ satisfy
\[
\sum_{k=0}^{b-1}F_{bn+k}^{\left(b\right)}=2F_{bn+\left(b-1\right)}^{\left(b\right)}+F_{bn+\left(b-2\right)}^{\left(b\right)}-F_{n}^{\left(b\right)}
\]
and, for $0\le p\le q<b-3,$
\[
\sum_{k=p}^{q}F_{bn+k}^{\left(b\right)}=F_{bn+q+2}^{\left(b\right)}-F_{bn+p+1}^{\left(b\right)}.
\]
\end{thm}
\begin{proof}
We use the identity
\[
\sum_{k=0}^{n}F_{k}=F_{n+2}-1
\]
and, for $0\le p\le b-1,$
\[
F_{bn+p}^{\left(b\right)}=F_{p}F_{n}^{\left(b\right)}
\]
to deduce
\begin{align*}
\sum_{k=0}^{b-1}F_{bn+k}^{\left(b\right)} & =\sum_{k=0}^{b-1}F_{k}F_{n}^{\left(b\right)}=F_{n}^{\left(b\right)}\left(F_{b+1}-1\right)\\
 & =F_{n}^{\left(b\right)}\left(F_{b}+F_{b-1}-1\right)=F_{n}^{\left(b\right)}\left(F_{b-2}+2F_{b-1}-1\right)\\
 & =F_{bn+b-2}^{\left(b\right)}+2F_{bn+b-1}^{\left(b\right)}-F_{n}^{\left(b\right)}.
\end{align*}

Moreover, for $0\le p\le q<b-3,$
\begin{align*}
\sum_{k=p}^{q}F_{bn+k}^{\left(b\right)} & =\sum_{k=0}^{q}F_{bn+k}^{\left(b\right)}-\sum_{k=0}^{p-1}F_{bn+k}^{\left(b\right)}\\
 & =F_{n}^{\left(b\right)}\sum_{k=0}^{q}F_{k}-F_{n}^{\left(b\right)}\sum_{k=0}^{p-1}F_{k}\\
 & =F_{n}^{\left(b\right)}\left(F_{q+2}-1\right)-F_{n}^{\left(b\right)}\left(F_{p+1}-1\right)\\
 & =F_{bn+q+2}^{\left(b\right)}-F_{bn+p+1}^{\left(b\right)}.
\end{align*}
\end{proof}

\begin{prop}
As a consequence of Cassini's identity
\[
F_{q}^{2}-F_{q+1}F_{q-1}=\left(-1\right)^{q},
\]
we deduce
\[
\left(F_{bn+q}^{\left(b\right)}\right)^{2}-F_{bn+q+1}^{\left(b\right)}F_{bn+q-1}^{\left(b\right)}=\left(-1\right)^{q}\left(F_{n}^{\left(b\right)}\right)^{2}
\]
and from its more general version
\[
F_{q}^{2}-F_{q+r}F_{q-r}=\left(-1\right)^{q-r+1}F_{r-1}^{2},
\]
we have
\[
\left(F_{bn+q}^{\left(b\right)}\right)^{2}-F_{bn+q+r}^{\left(b\right)}F_{bn+q-r}^{\left(b\right)}=\left(-1\right)^{n-r+1}\left(F_{bn+r-1}^{\left(b\right)}\right)^{2}.
\]

\end{prop}
%The following table shows the Fibonacci numbers $F_{n}^{\left(b\right)}$
%for all  bases $b$ between $2$ (first row) and $9$ (last row) and $n$ varying
%from $0$ to $10$ in each row. 
%\[
%\left(\begin{array}{ccccccccccc}
%1 & 1 & 2 & 1 & 3 & 2 & 3 & 1 & 4 & 3 & 5\\
%1 & 1 & 2 & 3 & 2 & 2 & 4 & 3 & 4 & 7 & 5\\
%1 & 1 & 2 & 3 & 5 & 4 & 3 & 3 & 6 & 5 & 5\\
%1 & 1 & 2 & 3 & 5 & 8 & 8 & 6 & 4 & 5 & 9\\
%1 & 1 & 2 & 3 & 5 & 8 & 13 & 15 & 13 & 8 & 6\\
%1 & 1 & 2 & 3 & 5 & 8 & 13 & 21 & 27 & 27 & 19\\
%1 & 1 & 2 & 3 & 5 & 8 & 13 & 21 & 34 & 47 & 53\\
%1 & 1 & 2 & 3 & 5 & 8 & 13 & 21 & 34 & 55 & 80\\
%1 & 1 & 2 & 3 & 5 & 8 & 13 & 21 & 34 & 55 & 89
%\end{array}\right)
%\]
\subsection{Another definition of base-$2$ Fibonacci numbers}
The extension of Fibonacci numbers defined in \eqref{4.1} in the case $b=2$ gives the uninteresting sequence
\[
F_n^{\left(2\right)} =1,\,\, \forall n\ge0,
\]
so that we propose the study of a slightly modified version of it.
Assume we define now the modified base-$2$ Fibonacci numbers as follows:
\[
\tilde{F}^{\left(2\right)}_{n} = \sum_{k=0}^{n} \binom{n-k}{k}_b
\]
so that we do not impose the digitally dominance on the summation index.

The first values of this new sequence, starting from $n=0,$ are
\[
1,1,2,1,3,2,3,1,4,3,5,\dots
%\tilde{F}_0^{\left(2\right)}=1,\,\, \tilde{F}_1^{\left(2\right)}=1,\,\, \tilde{F}_2^{\left(2\right)}=2,\,\, \tilde{F}%_3^{\left(2\right)}=1,\,\, \tilde{F}_4^{\left(2\right)}=3,\,\, F_5^{\left(2\right)}=2,\,\, F_6^{\left(2\right)}=3,\,\, %F_7^{\left(2\right)}=1,\,\, F_8^{\left(2\right)}=4,\,\, F_9^{\left(2\right)}=3,\,\, F_{10}^{\left(2\right)}=5.
\]
We recognize the first entries of Stern's diatomic sequence
$A002487$ defined by
\[
a_{0}=0,\thinspace\thinspace a_{1}=1\thinspace\thinspace\text{and}\thinspace\thinspace
a_{2n}=a_{n},\thinspace\thinspace\thinspace a_{2n+1}=a_{n}+a_{n+1},\thinspace\thinspace n\ge 1.
\]

\begin{thm}
The modified base-$2$ Fibonacci numbers $\tilde{F}_{n}$ are
\[
\tilde{F}_{n}^{\left(2\right)}=a_{n+1}
\]
where $\left\{ a_{n}\right\} $ is Stern's diatomic sequence.
\end{thm}
\begin{proof}
Since the base-$2$ binomial coefficient $\binom{n}{k}_{2}$ equals
$0$ or $1$ according to the parity of $\binom{n}{k},$ the modified base-$2$ Fibonacci
number
\[
\tilde{F}_{n}^{\left(2\right)}=\sum_{k\ge0}\binom{n-k}{k}_{2}
\]
is the number of odd binomial coefficients $\binom{n-k}{k}.$ 
It was identified by Carlitz who showed in \cite{Carlitz} that
this number $\theta_{0}\left(n\right)$ (in Carlitz notation) satisfies
\[
\theta_{0}\left(2n+1\right)=\theta_{0}\left(n\right),\thinspace\thinspace\thinspace\theta_{0}\left(2n\right)=\theta_{0}\left(n\right)+\theta_{0}\left(n+1\right)
\]
with $\theta_{0}\left(0\right)=1$ and $\theta_{0}\left(1\right)=1,$
hence $\theta_{0}\left(n\right)=a_{n+1}.$
\end{proof}

Another proof is obtained remarking that the modified base-$2$ binomial coefficients coincide with the usual binomial coefficients $\mod 2,$ and using the formula by S. Northshield \cite[Thm4.1]{Northshield} 
\[
\sum_{2i+j=n} \left(\binom{i+j}{i}\mod 2\right)  = a_{n+1}.
\]
Carlitz gives in \cite{Carlitz} a combinatorial interpretation of $a_{n+1}=\tilde{F}_{n}^{\left(2\right)}$ as the number
of hyperbinary representations of $n$, i.e the number of ways of writing $n$  as a sum of powers of $2$, each power being used at most twice.\\

\section{The base-b exponential}
\label{Exponential}
Analogously to the classical exponential function, we are interested in studying a modified exponential function
\begin{equation}\label{5.1}
e_b(x,w) := \sum_{k=0}^{\infty}\frac{x^{s_b(k)}}{(k!)_b}w^k.
\end{equation}
We note that without the weighting from $w^k$ this series has a zero radius of convergence. We can relate this to the classical exponential function, but require the following formula about the upper incomplete gamma function $\Gamma(a,z)$ \cite[(8.2.2)]{NIST:DLMF}, where
\begin{equation}\label{5.2}
\Gamma(a,z):=\int_{z}^{\infty} t^{a-1}e^{-t}dt.
\end{equation}
From \cite[(8.4.8)]{NIST:DLMF}, we have 
\begin{equation}\label{5.3}
\sum_{k=0}^{n}\frac{z^k}{k!} = e^z \frac{\Gamma(n+1,z)}{n!}.
\end{equation}

\begin{thm}
The base-$b$ exponential function satisfies
\begin{equation}\label{5.5}
e_b(x,w) = \exp\left(x\sum_{i=0}^{\infty}w^{b^i}\right) \prod_{i=0}^{\infty} \left(\frac{\Gamma(b,xw^{b^i})}{(b-1)!}\right) \simeq \exp\left(x\sum_{i=0}^{\infty} w^{b^i}\right) \simeq e^{xw+xw^{b}}
\end{equation}
for $w$ close to 0.
\end{thm}
\begin{proof}

Making the substitution $x \rightarrow w$ in Theorem \ref{genthm}, then taking $f(k_i,i) = \frac{x^{k_i}}{k_i!}$ while using \eqref{5.3} to simplify the partial sums as below completes the proof:
\begin{equation}\label{5.7}
\prod_{i=0}^{\infty}\sum_{k=0}^{b-1} f(k,i) = \prod_{i=0}^{\infty}\sum_{k=0}^{b-1} \frac{(xw^{b^i})^{k}} {k!} = \prod_{i=0}^{\infty} \exp\left(xw^{b^i}\right) \frac{\Gamma(b,xw^{b^i})}{(b-1)!} = \exp\left(\sum_{i=0}^{\infty} xw^{b^i}\right) \prod_{i=0}^{\infty} \left(\frac{\Gamma(b,xw^{b^i})}{(b-1)!}\right).
\end{equation}
Equating these representations proves the first part of the theorem. The behavior of our base-$b$ exponential then depends on the series $\sum_{i=0}^{\infty} w^{b^i}$ and an infinite gamma product. 
We now prove the first approximation. Looking at the infinite product term on the right-hand side shows that, even for small values of $b$ and $|x|<1$ it can be approximated by the indicator function of the interval $\left[0,1\right]$
\begin{equation}\label{5.9}
f\left(w\right)=\begin{cases}
1, & 0\le w \le 1\\
0, & w>1
\end{cases}
\end{equation}
This can be explained as follows: the function
\begin{equation}\label{5.10}
w\mapsto\frac{\Gamma\left(b,xw^{b^{i}}\right)}{\Gamma\left(b\right)}
\end{equation}
is strictly decreasing over $\left[0,1\right]$ from 1 at $w=0,$
to $0<\frac{\Gamma\left(b,x\right)}{\Gamma\left(b\right)}<1$ at $w=1.$ Moreover,
as \textbf{$b$ }increases, the ratio $\frac{\Gamma\left(b,x\right)}{\Gamma\left(b\right)}$
increases to 1.
For $0<w<1,$ $w^{b^{i}}$ is close to $0$ for $i$ large, so that
using the asymptotic expansion \cite[(8.7.3)]{NIST:DLMF}
\begin{equation}\label{5.11}
\Gamma\left(b,z\right)=\Gamma\left(b\right)+z^{b}\left(-\frac{1}{b}+\frac{z}{1+b}+\dots\right)
\end{equation}
we obtain
\begin{equation}\label{5.12}
\frac{\Gamma\left(b,xw^{b^{i}}\right)}{\Gamma\left(b\right)}=1-\frac{x^bw^{b^{i+1}}}{\Gamma\left(b+1\right)}+\frac{x^{b+1}w^{b^{i+1}+b^{i}}}{\left(b+1\right)\Gamma\left(b\right)}+\dots
\end{equation}
As a consequence, for $0\le w<1,$ $\sum w^{b^{i+1}}$ is convergent
so that $\prod_{i}\left(1-\frac{xw^{b^{i+1}}}{\Gamma\left(b+1\right)}\right)$
and $\prod_{i}\frac{\Gamma\left(b,xw^{b^{i}}\right)}{\Gamma\left(b\right)}$
are convergent. Moreover, from
\begin{equation}\label{5.13}
\log\Gamma\left(b,xw^{b^{i}}\right)-\log\Gamma\left(b\right)\simeq-\frac{1}{b}x^bw^{b^{i+1}},
\end{equation}
we deduce, for $0\le x<1,$
\begin{equation}\label{5.14}
\left|\sum\log\Gamma\left(b,xw^{b^{i}}\right)-\log\Gamma\left(b\right)\right|=\frac{1}{b}\sum_{i\ge0}x^bw^{b^{i+1}}\le\frac{1}{b}\sum_{i\ge0}w^{b^{i+1}}\le\frac{1}{b}\frac{w^{b}}{1-w}
\end{equation}
which goes to $0$ as $b\to\infty,$ so that the left-hand side goes
to $0$ as well. This yields the first approximation for $x\in\left[0,1\right]$, which can in turn be approximated as
\begin{equation}
e_b(x,w)\simeq e^{xw+xw^{b}}
\end{equation}
since for $w\in\left[0,1\right)$,
\begin{equation}
e^{w^{b^{2}}}\ll e^{w^{b}}.
\end{equation}
\end{proof}

For $b=2$, we obtain the following result.
\begin{cor}
The base $2$ exponential function satisfies
\begin{equation}\label{5.8}
e_2(x,w) = \prod_{i=0}^{\infty}\left(1+xw^{2^i}\right).
\end{equation}
\end{cor}

%%\end{proof}

Analogously to the classical identity $e^xe^y=e^{x+y}$, we have a similar convolution identity for the base-$b$ exponential.
\begin{thm}
\begin{equation}
e_b(x,w)\star e_b(y,w) = e_b(x+y,w),
\end{equation}
where 
\begin{equation}
\sum_{k=0}^{\infty}a_kx^{s_b(k)}\star \sum_{l=0}^{\infty}b_ly^{s_b(l)} = \sum_{n=0}^{\infty}\sum_{k\leq_b n}a_k b_{n-k}x^{s_b(k)}y^{s_b(n-k)}.
\end{equation}
\begin{proof}
The theorem follows from taking $a_k=\frac{w^k}{(k!)_b}$ and $b_l=\frac{w^l}{(l!)_b}$ and applying the digital binomial theorem \eqref{2.4}. 
\end{proof}
\end{thm}
The $\star$ operation forms an analogue of multiplication for formal power series which instead applies to series with involving powers of $s_b(k)$, where the convolution of two power series is defined as
\begin{equation}
\sum_{k=0}^{\infty}a_kx^{k}\star \sum_{l=0}^{\infty}b_ly^{l} = \sum_{n=0}^{\infty}\sum_{k\leq n}a_k b_{n-k}x^{k}y^{n-k}.
\end{equation}
We note that the $\star$ operator must manually impose the condition $k\leq_bn$ because of the basic result that
\begin{equation}
\binom{n}{k}_b \neq \frac{(n!)_b}{(k!)_b(n-k)!_b},
\end{equation}
since the left hand side is zero for $k$ not digitally dominated by $n$, while the right hand side is well-defined and in general non-zero for such $k$.
\bibliographystyle{plain}
%\bibliography{tvwref}

\end{document}